\documentclass[preprint,12pt]{elsarticle}
\usepackage{amsmath, amsfonts, amssymb, amsthm}
\usepackage{fullpage}
\usepackage[ruled,vlined]{algorithm2e}
\nocaptionofalgo
\usepackage{graphicx}
\usepackage{comment}

\newtheorem{thm}{Theorem}
\newtheorem{lem}[thm]{Lemma}
\newtheorem{cor}[thm]{Corollary}
\newtheorem{prop}[thm]{Proposition}
\newtheorem{conj}[thm]{Conjecture}
\theoremstyle{definition}
\newtheorem{defn}[thm]{Definition}
\newtheorem{prob}{Problem}

\newcommand{\st}{\;|\;}
\newcommand{\es}{ {\varnothing}}

\newcommand{\bi}{\begin{itemize}}
\newcommand{\ei}{\end{itemize}}
\newcommand{\be}{\begin{enumerate}}
\newcommand{\ee}{\end{enumerate}}
\newcommand{\bc}{\begin{center}}
\newcommand{\ec}{\end{center}}
\newcommand{\bt}{\begin{tabular}}
\newcommand{\et}{\end{tabular}}
\newcommand{\ba}{\begin{array}}
\newcommand{\ea}{\end{array}}
\newcommand{\bp}{\begin{proof}}
\newcommand{\ep}{\end{proof}}

\newcommand{\ra}{\rightarrow}

\newcommand{\N}{\mathbb N}

\newcommand{\LL}{\left[ {L \atop L} \right]}
\newcommand{\LN}{\left[ {L \atop N} \right]}
\newcommand{\LX}{\left[ {L \atop X} \right]}

\newcommand{\NL}{\left[ {N \atop L} \right]}
\newcommand{\NN}{\left[ {N \atop N} \right]}
\newcommand{\NX}{\left[ {N \atop X} \right]}

\newcommand{\XL}{\left[ {X \atop L} \right]}
\newcommand{\XN}{\left[ {X \atop N} \right]}
\newcommand{\XX}{\left[ {X \atop X} \right]}

\newcommand{\RR}{\left[ {R \atop R} \right]}
\newcommand{\II}{\left[ {I \atop I} \right]}
\journal{Journal of Algebra: Computational Algebra}

\begin{document}

\begin{frontmatter}

\title{Counting elements and geodesics\\ in Thompson's group~$F$}

 \author{Murray Elder\corref{cor1}}
 \ead{murrayelder@gmail.com}
\cortext[cor1]{Corresponding author}
 \address{School of Mathematics and Physics, University of Queensland, Brisbane, Australia}

 \author{\'Eric Fusy}
 \ead{fusy@lix.polytechnique.fr}
 \address{LIX, \'Ecole Polytechnique, Paris, France}
 
  \author{Andrew Rechnitzer}
 \ead{andrewr@math.ubc.edu}
 \address{Department of Mathematics, University of British Columbia, Vancouver,
Canada}

\begin{abstract}
We present two quite different algorithms to compute the number of elements in the sphere 
of radius $n$ of Thompson's group~$F$ with standard generating set. The first of these
requires exponential time and polynomial space, but additionally computes the number of
geodesics and is generalisable to many other groups.

The second algorithm requires polynomial time and space and allows us to compute the size
of the spheres of radius $n$ with $n \leq 1500$. Using the resulting series data we find
that the growth rate of the group is bounded above by $2.62167\ldots $. This is very close
to Guba's lower bound of $\tfrac{3+\sqrt{5}}{2}$ \cite{Guba2004}. Indeed, numerical
analysis of the series data strongly suggests that the growth rate of the group is
exactly~$\tfrac{3+\sqrt{5}}{2}$.
\end{abstract}


\end{frontmatter}

\section{Introduction}\label{intro}
Let $G$ be a group with finite generating set $X$. Recall that $f:\N\ra \N$ is the {\em
(spherical) growth function} for $(G,X)$ if $f(n)$ is the number of elements in the sphere
of radius $n$ of the corresponding Cayley graph. Define the {\em (spherical) geodesic
growth function} to be $g:\N\ra \N$ where $g(n)$ is the number of all geodesics of length
$n$ in the Cayley graph \cite{GrigTat}. In this article we give two quite different
algorithms to compute the growth function of Thompson's group~$F$. The first, Algorithm~A,
applies to a range of groups, and computes both $f(n)$ and $g(n)$. It runs in exponential
time and polynomial space, and is implemented  to compute the first 23 terms of both
functions with moderate computer resources. This algorithm is based on a random
sampling algorithm developed by the third author and van Rensburg for problems in lattice
statistical mechanics \cite{GARM}.

The second algorithm, Algorithm~B, is specific to Thompson's group~$F$, and computes
$f(n)$ in  polynomial time and space. It is based on the {\em forest diagram}
representation of elements of~$F$, introduced by Belk \& Brown \cite{BelkBrown2003}, and
the associated length formula, which itself is based on work of  Fordham 
\cite{Fordham2003}. We implemented this algorithm and were able to compute the first 1500
terms of $f(n)$, again with moderate computer resources. This data enables us to obtain an
upper bound of $2.62167\ldots$ on the growth rate of~$F$, which differs by only $0.15\%$
from the lower bound of $\frac{3+\sqrt{5}}{2}$ obtained by Guba  \cite{Guba2004}. 
Based on this and other numeric evidence, we
conjecture that Guba's bound is indeed the correct growth rate. Both sets of data (the
growth up to 1500 and geodesic growth up to 22) are published on the Online Encyclopedia
of Integer Sequences as A156945 and A156946 \cite{OEIS}.

Note that when the spherical (geodesic) growth is exponential, the
growth rate coincides with the usual growth and geodesic growth rates
which count the number of elements (geodesics) of length at most $n$ rather than exactly
$n$.

Many researchers have made an attempt to compute the growth of Thompson's group.
Guba \cite{Guba2004}  used a brute force technique to compute the first 9 terms of the
growth series. This was extended to the first 13 by Burillo, Cleary and Weist
\cite{Burillo2007}. Matucci considered a system of recurrences obtained by considering
forest diagrams in his thesis  \cite{Matucci}, which turned out to be quite complicated.
It is possible that his approach could be used to
compute the growth series but this does not appear to have been pursued.
Our own approach is to set up an algorithm (being Algorithm~B) that enumerates the
number of elements in the sphere of radius $n$ by computing the number of forest diagrams
of weight $n$. Using this we have been able to compute a large number of terms. One key
ingredient in our algorithm is the encoding of forest diagrams by labelling both the gaps
between leaves (as Belk \& Brown do) as well as the {\em internal nodes} of trees, which
is based on an encoding of binary trees described in~\cite{Felsner2008}.

Unfortunately, we have not been able to leverage our algorithm into a closed form
expression for the growth function or the corresponding generating function. However, an
analysis of our data indicates that the series does not correspond to a
simple\footnote{\emph{ie} satisfying an equation of low order which
has coefficients of low degree.} rational, algebraic or differentiably-finite generating
function (satisfying a linear ordinary differential equation with polynomial coefficients
--- see \cite{StanleyECV2}). It is entirely possible that the generating function lies
outside these classes of functions and that no reasonable 
 closed form
solution exists. See \cite{MBMMP_2000, MBMMP_2003} for examples of
problems with similar (but far simpler) recurrences that have surprisingly complicated
solutions. 

The paper is organised as follows. In Section \ref{sec:AlgA} we describe the first algorithm, 
which computes in exponential time 
the coefficients of the growth and geodesic growth functions of any group with an efficient solution to
a certain {\em geodesic problem}. We apply this to Thompson's group~$F$ and display our
results. In Section \ref{sec:AlgB} we describe the second, polynomial time algorithm which
we have used to compute the first 1500 terms of the growth function of~$F$. We describe
the forest
diagram construction, how it leads to our enumeration algorithm,  and give the
results. We then prove an upper bound for the growth rate. In Section \ref{sec:outlook} we
summarise our findings, and in Appendix~\ref{appendix} we give more detailed pseudocode
for our algorithms.

We work with the presentation
\[\left\langle \;
x_0, \; x_1 \;\; | \;\; [x_0x_1^{-1},x_0^{-1}x_1x_0],[x_0x_1^{-1}, x_0^{-2}x_1x_0^2]\;
\right\rangle
\]
for Thompson's group~$F$.   We refer the reader to \cite{CFP} and \cite{CombCT}
for an introduction to Thompson's group. 

The authors thank Jos\'e Burillo, Sean Cleary, Martin Kassabov, Manuel Kauers, Francesco
Mattucci, Buks van~Rensburg and Jennifer Taback  for fruitful discussions and ideas, and
the anonymous reviewer for their careful reading and very helpful feedback.

\section{Algorithm A}\label{sec:AlgA}

We begin with the following lemma, due to the third author and van Rensburg in \cite{GARM},  
 which arises in the context of randomly sampling
self-avoiding walks and polygons using a generalisation of the Rosenbluth
method \cite{PMMC, Rosenbluth}. In particular, it allows us to enumerate objects
without having to use a unique construction.
\begin{defn} Let $G$ be a group with finite generating set $X$. Given a word
  $w$ in the generators of the group define
  \begin{align*}
  d_-(w) &= \{ x \in X \st \ell(w)>\ell(w x) \} \\
  d_+(w) &= \{ x \in X \st \ell(w)<\ell(w x) \} \\
  d_0(w) &= \{ x \in X \st \ell(w)=\ell(w x) \},
  \end{align*}
  where $\ell(w)$ is the geodesic length of the element represented by $w$.
  These are the subsets of the generators that shorten, lengthen and do no
change the geodesic length. Note that if the group only has relations of even length (as
is the case for~$F$) then
$d_0(w) \equiv \emptyset$.
\end{defn}

\begin{lem}[From \cite{GARM}]\label{lem:GARM}
  Let $\Gamma_n$ be the set of all geodesic words of length $n$. Given a word
  $w$ from this set, let $w_i$ be its prefix of length $i$. The size of the
  sphere of radius  $n$ (ie the  number of distinct elements whose geodesic
  length is $n$) is given by
  \begin{align*}
  | S(n) | 
  &= \sum_{w \in \Gamma_n} \prod_{i=1}^n \frac{1}{|d_-(w_i)|}.
  \end{align*}
\end{lem}
\begin{proof}
  We prove this result by induction on $n$. When $n=1$, the set of geodesics
  is just the set of generators and so expression is true.

  Let $g$ be an element of $S(n)$ and let $\Gamma_n(g)$ be the set of geodesic
  paths from the identity to $g$.  It then suffices to show that
  \begin{align*}
  \sum_{w \in \Gamma_n(g)} \prod_{i=1}^n \frac{1}{|d_-(w_i)|} &= 1.
  \end{align*}
  Every geodesic ending at $g$ can be written as the product of a geodesic in
  $\Gamma_{n-1}$ and a generator. Thus we can write
  \begin{align*}
  \sum_{w \in \Gamma_n(g)} \prod_{i=1}^n \frac{1}{|d_-(w_i)|}
  &= 
  \frac{1}{|d_-(w)|} \sum_{w \in \Gamma_n(g)}
  \prod_{i=1}^{n-1} \frac{1}{|d_-(w_i)|} \\
  &= \sum_{x \in X} \frac{1}{|d_-(w)|} 
  \sum_{v \in \Gamma_{n-1}(gx^{-1})}
  \prod_{i=1}^{n-1} \frac{1}{|d_-(v_i)|}
  \end{align*}
  The inner sum is zero when the set $\Gamma_{n-1}(gx^{-1})$ is empty. The
  induction hypothesis implies that the inner sum is equal to $1$ when the set
  is non-empty (ie when $\ell(g x^{-1}) = n-1$). Exactly $|d_-(w)|$ of the
  sets are non-empty and so the result follows.
\end{proof}

\subsection{Description of the algorithm}

The above lemma allows us to compute the size of $S(n)$ much more efficiently
than brute-force methods. The time complexity is proportional to the size of
$\Gamma_n$ and the time to compute $d_\pm(w)$. The memory required is
significantly reduced; we only require space to compute $d_\pm(w)$ and the
current geodesic word. This is substantially better than brute force methods; we
do not have to store normal forms for all the elements in $S(0), \dots S(n)$.

Suppose that $G$ has an inverse closed generating $X$ set of size $m$, and that there is
an algorithm to compute $d_-(u),d_+(u)$ for any geodesic $u$ in this generating set. Fix
an order on the generators of $X$ (this also fixes a lexicographic order on the set of
all geodesic words). Then the following algorithm computes the number of elements of
length $n$.

\vspace{3mm}

\noindent \hrulefill 

\noindent \textbf{Algorithm A.} Outputs the number of elements in the $n$th sphere.  

\vspace{-3mm}
\noindent \hrulefill 
\vspace{3mm}

\noindent Set a counter to 0.

\vspace{3mm}

\noindent Starting from the empty word, run through all geodesics of length~$n$.

\vspace{3mm}

\noindent Given the current geodesic, $u$, compute $|d_-(u_i)|$ for each length $i$
prefix of $u$, and add $\prod_{i=1}^n \frac{1}{|d_-(u_i)|}$ to the counter.

\vspace{3mm}

\noindent By Lemma \ref{lem:GARM} the final value of the counter is the number of elements
in the sphere of radius~$n$.

\vspace{-3mm}
\noindent \hrulefill

\vspace{3mm}

There are a number of ways of running through the list of all geodesics; perhaps the
simplest is to traverse all the geodesics in lexicographic order using a recursive
depth-first algorithm (see the pseudocode for \texttt{ComputeGeod()} and
\texttt{ComputeSphere()} in Appendix~\ref{app A}). It can also be done using a slightly
more involved iterative procedure that given a geodesic finds the next geodesic in the
lexicographic ordering (see the pseudocode for \texttt{NextGeod()} also in
Appendix~\ref{app A}).

Note that the above algorithm is easily modified to obtain the number of geodesics
of length $n$: rather than adding $\prod_{i=1}^n \frac{1}{|d_-(u[i])|}$ to the counter
each time, simply add~$1$.

\begin{prop}
The time complexity of Algorithm  A is $O(|\Gamma_n|)$,  multiplied by the time to
compute $d_-,d_+$ at each step. The space complexity is linear plus the space required to
compute $d_-,d_+$.
\end{prop}
\bp
As we run through all geodesics $u$, we only need space to store the {\em current geodesic
word}, which requires linear space, plus the space required to compute $d_-,d_+$. Since we
must consider all geodesics, the time is proportional to the number of geodesics in
$|\Gamma_n|$, multiplied by the time to compute $d_-, d_+$ at each step.
\ep

So the time and space complexity of the algorithm depends on the complexity of computing
$d_+$ and $d_-$.  This naturally leads to the geodesic problems described in the next
section.

\subsection{Geodesic problems}
Algorithm A gives a memory-efficient procedure to compute growth and geodesic growth for
any group that has an efficient solution to the following problem: 

\begin{prob}
Given a geodesic $u$ and generator $x$, determine if $x\in d_+(u), d_0(u)$ or $d_-(u)$.
\end{prob}
Note that if all the relators in $R$  have even length then these are the only two
possibilities, since if $x\in d_0(u)$ means there must be an odd length relator.
 So the problem turns into the decision problem:
\begin{prob}
Given a geodesic $u$ and generator $x$, decide if $x\in d_+(u)$.
\end{prob}
  If $(G,X)$ has a polynomial time and space algorithm to answer Geodesic Problem 1,
then Algorithm A computes both $f(n)=|S(n)|$ and $g(n)=|\Gamma(n)|$ in exponential
time, but polynomial space (assuming $g$ grows exponentially). In \cite{ElderGeodesics}
the authors consider these problems in more detail, and prove that they imply a solvable
word problem for groups whose set of relators is countably enumerable. Note that if a
group has an efficient solution to its word problem, then a naive brute force computation
of $f(n)$ and $g(n)$ would still require an exponential amount of memory (assuming the
functions $f$ and $g$ are exponential), so a polynomial space algorithm is a marked
improvement on this.

\subsection{Thompson's group~$F$}

For Thompson's group~$F$, one can use the geodesic length formula of Fordham
\cite{Fordham2003} to compute the length of  any word in polynomial time and space.
There are now other similar formulations, and we note those due to Guba \cite{Guba2004}
and Belk \& Brown \cite{BelkBrown2003}. We chose to base both of our algorithms on the
forest diagrams of Belk \& Brown as we found them the easiest to implement. We note that
there is a simple bijection between this representation and binary tree pairs. It is
possible to develop similar algorithms based on the Fordham and Guba formulations, but we
have not pursued these possibilities because we believe that they will give very similar
results.

We start our description of the algorithm by first explaining the forest-diagram length
formula of Belk \& Brown (we refer the reader to the original paper \cite{BelkBrown2003}
for a fuller explanation). In Figure~\ref{fig caretdef} we describe some of the
terminology associated with this formula.
\begin{figure}[h!]
 \centering
 \includegraphics[scale=0.5]{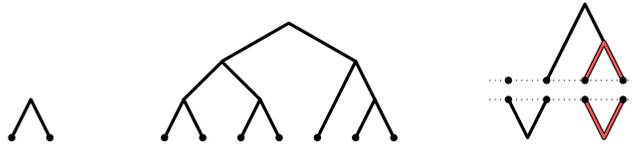}
 \caption{(left) A caret is a inner node of a binary tree together with the two out-going
edges. (middle) A binary tree consisting of 6 carets and so 6 inner nodes and 7 leaves.
(right) A pair of forests with a pair of common carets highlighted; such a diagram is not
admissible.}
 \label{fig caretdef}
\end{figure}

Consider the forest diagram in Figure~\ref{fig bbexample} (which is Example~4.2.8 from
\cite{BelkBrown2003}).
\begin{figure}[h!]
 \centering
 \includegraphics[scale=0.5]{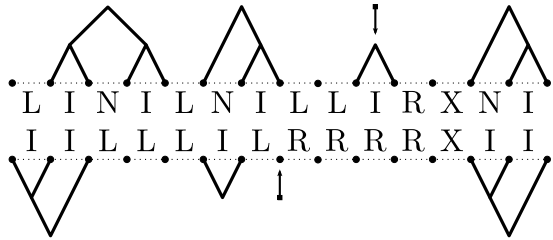}
 \caption{A forest diagram. The gaps are labelled according to a simple set of
rules. The length of the corresponding element of Thompson's group is computed
by considering the label-pairs in each column. In this case the element has
length $[2,2,2,2,2,4,2,1,1,1,2,2,4,2] = 29$.}
 \label{fig bbexample}
\end{figure}
It consists of two rows
of forests of binary trees, above and below two rows of labelled gaps. One (possibly
empty) tree in each forest is pointed and the diagram cannot contain ``common carets''
(see Figure~\ref{fig caretdef}). A pair of carets in a given column are considered
to be common carets, if they are both immediate neighbours of the gaps. Further, the
extreme left and right columns of the diagram cannot be empty. Each element of~$F$
corresponds uniquely to such  a  forest diagram with no common carets.
 
The labels of the gaps are determined by the structure of the forests. We found it
convenient to make a slight modification of the original labelling rules; we label gaps by
applying the following rules in order:
\renewcommand{\theenumi}{(\roman{enumi})}
\begin{enumerate}
 \item ``L'' if it is left of the pointer and exterior (outside a tree),
 \item ``N'' if it is immediately left of a caret and interior (inside a tree),
 \item ``X'' if it is immediately left of a caret and exterior,
 \item ``R'' if it is right of the pointer and exterior,
 \item ``I'' if it is interior.
\end{enumerate}
Belk \& Brown give ``N'' and ``X'' the same label, however the above labelling means
that ``L'', ``R'' and ``X'' are always exterior, while ``I'' and ``N'' are always
interior. The length of a group element represented by a
diagram $u$ is then
\begin{align}
  \ell(u) &= \sum_{c \in \mathrm{columns}} W(c)
\end{align}
where $W(c)$ is the weight of a column as given in Table~1. We refer to this sum as the
weight of a given forest diagram.
\begin{center}
\begin{table}[h!]
\large
\begin{center}
\begin{tabular}{|l||c|c|c|c|c|}
\hline
 $\circ$ & I & N & L & R & X \\
 \hline\hline
       I &  2 & 4 & 2 & 1 & 3 \\
       \hline
       N &  4 & 4 & 2 & 3 & 3 \\
       \hline
       L &  2 & 2 & 2 & 1 & 1 \\
       \hline
       R &  1 & 3 & 1 & 2 & 2 \\
       \hline
       X &  3 & 3 & 1 & 2 & 2 \\
       \hline
\end{tabular}
\end{center} 
\caption{The weight of a column is given by the above simple function of the
gap labels.}
\end{table}
\end{center} 
Our table differs from that in Belk \& Brown \cite{BelkBrown2003} also in that our table
entries include the contribution from the number of carets, whereas Belk \& Brown add this
contribution separately.

It should be  noted that labels are defined very locally; to label a given column we do
not need to know the labels of columns far away. While this is not crucial for
Algorithm~A, it is critically important for Algorithm~B as it allows us to build a
diagram with its labels column-by-column. This will be discussed further below.

\subsection{Implementing Algorithm A for Thompson's group~$F$}

Propositions 3.3.1 and 3.3.5 of \cite{BelkBrown2003} describe how multiplication by each
of the  generators and their inverses changes a forest diagram. Multiplying a reduced
forest by $x_0^{\pm 1}$ simply moves the top pointer one tree to the right or
left which can be done in constant time. Multiplying by $x_1$ adds a single
caret to the roots of two adjacent trees, then canceling common carets if they exist.
Multiplying by $x_1^{-1}$ either adds or deletes a caret in the bottom or top forests, and
then cancels common carets if they exist. Note that after multiplying by $x_1^{\pm 1}$ at
most a single common caret exists.

So a word of length $n$ produces a forest of at most $n$ carets, so we can store the
forest diagram for such a word in space $O(n)$. We can turn this into an algorithm to
compute the geodesic length of an element of~$F$ as follows:

\vspace{3mm}

\noindent \hrulefill 

\noindent \textbf{Algorithm} \textsc{GeodesicLengthF}.  Computes the geodesic length of an  
element in~$F$ 

\vspace{-3mm}
\noindent \hrulefill 

\noindent Input a word  of length $n$ in the generators $x_0^{\pm 1},x_1^{\pm 1}$.
\vspace{3mm}

\noindent Start with the empty forest diagram (no carets, and both pointers at the 0
position).

\vspace{3mm}

\noindent For each letter of $u$, redraw the current forest diagram by multiplying by this
letter. To do this one must scan through the stored diagram and possibly add or delete a
caret; this takes $O(n)$ operations.

\vspace{3mm}

\noindent  When all letters are read, compute the weight of the resulting diagram by
labelling it (which takes $O(n)$ operations) then read the geodesic length off from
the table (which takes $O(n)$ operations).

\vspace{-3mm}
\noindent \hrulefill 

\vspace{3mm}

\begin{prop}
One can compute the geodesic length of a word in~$F$ in time $O(n^2)$ and space
$O(n)$.
\end{prop}
\bp 
We store the forest diagram as a pair of two trees, and use depth-first search to run
through the trees. The size of the tree is $O(n)$ and it  takes $O(n)$ steps to
scan through and reduce a pair of common carets. So as we run through each letter of $u$
we perform $O(n)$ steps, so all together $O(n^2)$. The final step takes $O(n)$ to label
the forest diagram and $O(n)$ to compute the weight.
\ep

The above algorithm allows us to quickly compute the geodesic length of any element
of~$F$ (and so $d_\pm$); it formed the basis of our implementation of Algorithm~A. This
then allowed us to calculate both the growth function and geodesic growth function for
Thompson's group~$F$ up to length~$22$. See Table~3. Since~$F$ has exponentially many
geodesics (it has exponential growth) the time for Algorithm A is exponential, but space
is $O(n)$.

\begin{center}
\begin{table}[h!] 
\begin{center}
 \begin{tabular}{|c|r|r||c|r|r|}
 \hline
 $n$ & $|S(n)|$ & $|\Gamma_n |$ &
 $n$ & $|S(n)|$ & $|\Gamma_n |$ \\
 \hline\hline
0 & 1 & 1 & 12 & 431238 & 556932 \\
1 & 4 & 4 & 13 & 1180968 & 1588836 \\
2 & 12 & 12 & 14 & 3225940 & 4507524 \\
3 & 36 & 36 & 15 & 8773036 & 12782560 \\
4 & 108 & 108 & 16 & 23809148 & 36088224 \\
5 & 314 & 324 & 17 & 64388402 & 101845032 \\
6 & 906 & 952 & 18 & 173829458 & 286372148 \\
7 & 2576 & 2800 & 19 & 467950860 & 804930196 \\
8 & 7280 & 8132 & 20 & 1257901236 & 2255624360 \\
9 & 20352 & 23608 & 21 & 3373450744 & 6318588308 \\
10 & 56664 & 67884 & 22 & 9035758992 & 17654567968 \\
11 & 156570 & 195132 & &  &\\
\hline
 \end{tabular}
\end{center}
\caption{Enumeration of group elements and geodesics of Thompson's group~$F$.}
\end{table}
\end{center}

In \cite{Guba2004} Guba proves that the growth rate of the number of elements in
Thompson's group~$F$ is bounded below by   $\frac{3 + \sqrt{5} }{2} = 2.61803\dots$, which
also serves as a lower bound for the growth rate of geodesics. We can use our data to
obtain upper bounds for both growth and geodesic growth.

Since every element of length $n+m$ is the concatenation of some element of length $n$ and
some element of length $m$, it follows that $f(n+m)\leq f(n)f(m)$ and so the sequence
$f(n)$ is submultiplicative. The same is true for the sequence $g(n)$, which counts the
number of geodesics. Fekete's lemma \cite{Fekete} states that if a sequence $a(n)$ is
submultiplicative, then the sequence $a(n)^{1/n}$ is monotonically decreasing, which means
that for any fixed $n$, the number $a(n)^{1/n}$ is an upper bound on the limit of
$a(n)^{1/n}$. From this we obtain the following:

\begin{prop}
  The growth rate $\gamma:=\lim_{n \to \infty} |S(n)|^{1/n}$ for the growth function 
of Thompson's group~$F$ satisfies the bounds
\begin{align}
  \frac{3 + \sqrt{5} }{2} = 2.61803\dots
  & \leq \gamma \leq
  |S(22)|^{1/22} =2.8349398\ldots
\end{align}
  Similarly, the growth rate $\mu:=\lim_{n \to \infty} |\Gamma_n|^{1/n}$ for 
the geodesic growth function of  Thompson's group~$F$ satisfies the bounds
\begin{align}
  \frac{3 + \sqrt{5} }{2} = 2.61803\dots
  & \leq \mu \leq
  |\Gamma_{22}|^{1/22} = 2.92257\dots
\end{align}
\end{prop}
In the next section we will sharpen the upper bound for the growth of elements
considerably.

\section{Algorithm B}\label{sec:AlgB}
In this section we describe a second algorithm to compute the number of elements in the
sphere of radius $n$ in Thompson's group~$F$ with standard generating set; this
algorithm runs in polynomial time and space.

Recall the forest diagram definition from the previous section. We count forest diagrams
by their length using a column-by-column construction. Such constructions have been used
to great effect in combinatorics and statistical mechanics yielding closed form solutions
(for example \cite{Temperley, BousquetMelou}), polynomial time algorithms (for example
\cite{3choice}) and exponential time algorithms that are exponentially faster than brute
force methods (for example \cite{Enting1985}). The construction we use here gives rise to
a system of recurrences which is rather cumbersome. As such, we do not give these
recurrences explicitly, but instead iterate them using Algorithm~B.

An incomplete forest diagram can be thought of as a forest diagram with all rightmost columns deleted after some point.
 In order to append a column in a ``legal'' way we only need to keep track of the labels
of the rightmost column and two numbers which measure how far the top and bottom trees are
from being ``complete''. 

\subsection{Encoding and counting binary trees}
\label{ssec btrees}
Since the forest diagrams consist of a pair of sequences of binary trees, the
starting point for our enumeration algorithm is to find an efficient encoding
of binary trees. Binary trees are counted by the Catalan numbers and there is a large
number of encodings.
 We have chosen one that naturally reflects the column-by-column construction
used.

We start by labelling leaves and internal nodes of the trees by ``N'', ``I'',
``n'' and ``i'' as shown in Figure~\ref{fig node labels}. By convention we will not label the left-most leaf.
\begin{figure}[h!]
  \begin{center}
 \includegraphics[scale=0.5]{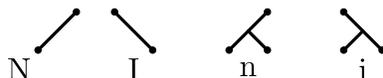}\\
  \end{center}
  \caption{We label leaves of a binary tree by ``N'' and ``I'' as
  shown. Similarly we label the internal nodes by ``n'' and ``i''. By
convention we always label the root node ``n''.}
\label{fig node labels}
\end{figure}

Reading these labels from left to right gives a word in the alphabet $\{ \mbox{n, N, i, I}
\}$ which starts with an ``n'' and alternates lower and upper case
letters. The tree in Figure~\ref{fig codeword} is encoded by the word ``nInNiInNnIiI''.
\begin{figure}[h!]
  \begin{center}
 \includegraphics[scale=0.6]{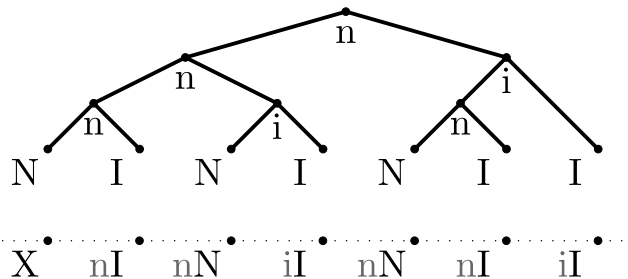}
  \end{center}
  \caption{The tree above is labeled as shown. This gives rise to the encoding
$\mbox{nInNiInNnIiI}$. For the purposes of the proof of Lemma~\ref{lem:rules} we 
prepend an N to this word. In our main algorithm, we will label the first
leaf by ``L'' or  ``X'', and delete all of the lower case labels.}
  \label{fig codeword}
\end{figure}

This labelling was used in  \cite{Felsner2008} and they proved that each tree has
a unique encoding. The following lemma is Theorem~4.4 in that paper with modified
notation.
\begin{lem}
  \label{lem:rules}
  A word, $w$, encodes a non-empty tree if and only if
\begin{enumerate}
\item it starts with ``n'',
\item ends with ``I'', 
\item alternates lowercase and uppercase letters, 
\item all of its prefixes $u$ satisfy $|u|_{\mbox{n}} + |u|_{\mbox{N}} \geq |u|_{\mbox{i}}
+ |u|_{\mbox{I}}$ (where $|u|_{\mbox{y}}$ denotes the number of occurrences of ``y'' in
$u$)
\item $|w|_{\mbox{n}} + |w|_{\mbox{N}} = |w|_{\mbox{i}} + |w|_{\mbox{I}}$.
\end{enumerate}
\end{lem}
\bp
For the purposes of the proof we prepend an initial ``N'' to all encoding words. Call a
word on $\{ \mbox{n, N, i, I} \}$ \emph{admissible} if it has an initial ``N''
and it satisfies the 5 conditions above when the initial ``N'' is deleted.

Consider the mapping $\psi_k$ that associates a binary tree of $k$ carets to
the word of length $2k+1$ obtained by starting with an ``N'' and then reading the labels
from left to right (see Figure~\ref{fig codeword}). We show by induction on $k$ that
$\psi_k$ is bijective onto admissible words of length $2k+1$.

For $k=1$, this is evident, there is a unique tree with one caret, which is mapped by
$\psi_1$ to the unique admissible word of length $3$, namely ``NnI''.

Assume that the bijection holds for $k\geq 1$ and we will show that it also holds for
$k+1$. First we check that the image $w$ of a tree $t$ with $k+1$ carets is admissible.
Let $t'$ be a tree obtained from $t$ by deleting two leaves adjacent to the same caret $c$
(so $c$ becomes a leaf in $t'$) and let $w'=\psi_k(t')$. As seen in
Figure~\ref{bij_proof}, $w$ is equal to $w'$ where an upper-case letter ``A'' is
replaced by the 3-letter factor ``NaI'', where ``a'' is the lower case of ``A''. Since
$w'$ is admissible (by induction), it follows that $w$ is also admissible.

\begin{figure}[h!]
  \begin{center}
 \includegraphics[scale=0.6]{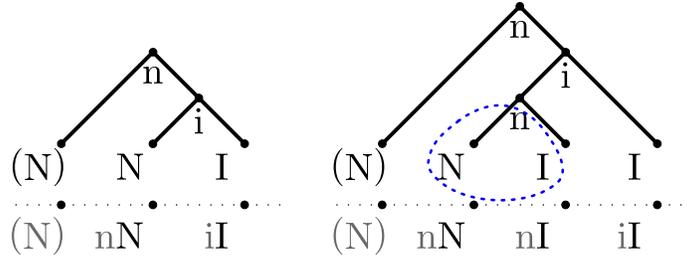}
  \end{center}
  \caption{The effect of expanding a leaf with label $\mbox{A} \in\{ \mbox{N, I} \}$ into
a caret is to replace ``A'' by ``NaI'' in the encoding word, where ``a'' is the lower case
of ``A''. }
  \label{bij_proof}
\end{figure}

Now to prove that $\psi_{k+1}$ is bijective consider an admissible word $w$ of length
$2k+3$ and we will show that it has a unique pre-image. Let $W$ be the word obtained from
$w$ by keeping only its upper-case letters. Since $W$ starts with ``N'' and ends with
``I'' it must contain a factor ``NI". This lifts to a 3-factor ``NaI'' in the word
$w$. Replacing this 3-factor by a single letter ``A'', where ``A'' is the upper-case of
``a'', we obtain a word $w'$ of length $2k+1$. It follows that $w'$ is also admissable.

Moreover one may verify (see Figure~\ref{bij_proof}) that any pre-image $t$ of $w$ must
be obtained from a pre-image $t'$ of $w'$ where the leaf corresponding to ``A'' is
replaced by a caret with two leaves below it. Since there is a unique pre-image of $w'$
(by induction), there is also a unique pre-image of $w$.
\ep

We can now count binary trees by counting admissable words in this alphabet.
Define the {\em size} of an even-length word to be its half-length.  
Following from condition~(iv) above we define the {\em excess} 
of a word, $w$, to be the quantity $\frac12 \left(|w|_N + |w|_n - |w|_I - |w|_i \right)$.
So appending ``nN'' increases the excess by $1$, appending ``nI'' or ``iN''
does not change the excess and appending ``iI'' decreases the
excess by $1$. Further, appending any of these pairs increases the size by $1$..
\begin{cor}\label{cor:incom}
A word encodes a nonempty binary tree if and only if its excess is zero.
\end{cor}
\bp
This follows from condition~(v) above.
 \ep

Given an even length word $u$ and its excess we can decide which pairs can
be legally added to the word --- in fact there is little restriction other than when the
excess is $0$, in which case one may not append $\mbox{iN}$ or $\mbox{iI}$
since this violates condition~(iv). Thus one can construct a recurrence satisfied by
$c_{\ell,h}$, the number of valid words of length $2\ell$ (size $\ell$) and excess $h$.
\begin{itemize}
 \item There is a single valid code word of size $1$, namely ``nI''. Hence $c_{1,0}=1$.
 \item Any word with excess equal zero can be extended by
  \begin{itemize}
  \item appending ``nN'' which makes its excess 1, or
  \item appending ``nI'' which leaves its excess 0.
  \end{itemize}
  \item Any word with excess strictly greater than zero can be extended by
  \begin{itemize}
  \item appending ``nN'' which increases its excess by 1, or
  \item appending ``nI'' which leaves  its excess unchanged, or 
  \item appending ``iN'' which leaves  its excess unchanged, or 
  \item appending ``iI'' which decreases its excess by 1.
  \end{itemize}
\end{itemize}
This translates into the following recurrence.
\begin{align}
 c_{1,0} &= 1 \\
 c_{\ell,h} &= \begin{cases}
             c_{\ell-1,0}+c_{\ell-1,1} & \mbox{if $h=0$}\\
             c_{\ell-1,h+1} + 2c_{\ell-1,h} + c_{\ell-1,h-1} & \mbox{otherwise}
            \end{cases}
\end{align}
While this recurrence can be solved explicitly, we show how it may solved iteratively in
order to help explain our main result. Notice that the algorithm that we use to iterate
this recurrence does not ``look backward'' and compute $c_{\ell,h}$ in terms of $c_{\ell-1,h}$
and $c_{\ell-1,h\pm1}$, rather it ``looks forward'' and determines which coefficients
$c_{\ell+1,h'}$ gain contributions from $c_{\ell,h}$. Of course this algorithm (and our
main result) could be rewritten to be ``backward looking'' but the authors feel that this
way follows more naturally from the construction.

\vspace{3mm}

\noindent \hrulefill

\noindent \textbf{Algorithm} \textsc{CountBinaryTrees}. Outputs the number of binary trees of a desired size. 

\vspace{-3mm}
\noindent \hrulefill 

\noindent Input maximum size (half-length) $N$ and set size $\ell=0$.

\vspace{3mm}

\noindent Set $\mathtt{words}(1,0)=1$ --- $\mathtt{words}(\ell,h)$ has to store the number of words of size $\ell$ and excess $h$.

\vspace{3mm}

\noindent At the current size $\ell$ run through all possible values of
the excess $h$.

\vspace{3mm}
\noindent If the current value of $h$ is zero then (because
we can append ``nN'' and ``nI'')
\begin{itemize}
 \item increment $\mathtt{words}(\ell+1,1)$ by $\mathtt{words}(\ell,0)$, and
 \item increment $\mathtt{words}(\ell+1,0)$ by $\mathtt{words}(\ell,0)$.
\end{itemize}
\noindent Otherwise (we can append ``nN'', ``nI'', ``iN'' and ``iI'')
\begin{itemize}
 \item increment $\mathtt{words}(\ell+1,h+1)$ by $\mathtt{words}(\ell,h)$, and
 \item increment $\mathtt{words}(\ell+1,h)$ by $2\mathtt{words}(\ell,h)$, and
 \item increment $\mathtt{words}(\ell+1,h-1)$ by $\mathtt{words}(\ell,h)$.
\end{itemize}

\vspace{3mm}

\noindent After running through all possible values of the excess, increment the size $\ell$ by $1$ 
 and output $\mathtt{words}(\ell,0)$ --- the number of valid complete words of the
current size.

\vspace{3mm}

\noindent Keep iterating until the desired size is reached.

\vspace{-3mm}
\noindent \hrulefill 
\vspace{3mm}

Notice that we must iterate through all values of the size and
excess in order to compute the final enumeration. Since the excess is at most the size
for each (possibly incomplete) forest diagram, this means that the
algorithm requires $O(n^2)$ time. Once we reach a given size we can forget all the
data from shorter lengths. In fact we only need to remember the current size and the
next size. This means the algorithm requires $O(n)$ space. 

For the main algorithm below, the labels ``I'' and ``N'' correspond precisely
to the labels used in the Belk \& Brown geodesic length formula. The lower case
labels can then be ignored and we instead keep track of the excess and note
that when appending an ``N'' the excess either stays constant or increases by
$1$, and when appending an ``I'' the excess either stays constant or decreases
by $1$. Note also that the size (which was the half-length when lower-case letters are
considered) coincides with the length when the lower-case letters are ignored, so we will
 use the terminology of length (instead of size) from now on.

\subsection{Transitions in the upper forest}
\label{ssec trans}
We now describe how to construct the upper half of the forest diagram column-by-column;
the lower half is encoded by an  identical procedure. We commence every forest with an
empty column labelled ``L''. The possible transitions from a column with a given label to
the next depend only on that label, the excess and whether one is to the
left or right of the pointer. Note that the excess can only be non-zero
if the label is either ``I'' or ``N''. We now describe this transitions in detail.

We will write the last column of the current upper forest as a triple of the
label, a flag indicating whether it is left or right of the pointer, and
the excess. For example, a diagram ending in an ``L'' must
be left of the pointer, and have excess 0 and so we denote it by
$\mathtt{(L,left,0)}$. Call this triple the state of the upper diagram.
\begin{figure}[h!]
  \begin{center}
 \includegraphics[scale=0.6]{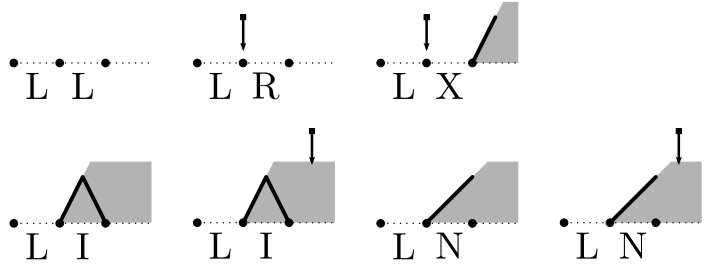}
  \end{center}
  \caption{Consider a diagram ending in ``L''. We can append new columns to the
end of this diagram as shown above. In particular, an ``L'' may be followed by
an ``L'' or ``R''. If we append an ``X'' then we must commence a tree at
the next column (indicated by the shaded region). Similarly we may start a
tree in exactly 4 different ways; either with an ``I'' or an ``N'' and with the
tree pointed or not.}
\label{fig Lappend}
\end{figure}
\begin{itemize}
\item If the upper forest ends in an ``L'' then it must be in state
$\mathtt{(L,left,0)}$. See Figure~\ref{fig Lappend}. We may append
\begin{itemize}
\item ``L'', which keeps the
diagram  in state $\mathtt{(L,left,0)}$,
\item ``R'', which means the pointer is now to the left of the current position, so the
diagram is now in state $\mathtt{(R,right,0)}$,
\item ``X'', so again the pointer has been passed, and the diagram is in state
$\mathtt{(X,right,0)}$,
\item ``I'', which starts a tree with codeword ``nI'' so the state becomes
$\mathtt{(I,left,0)}$,
\item ``I'' and put the pointer on the tree (just started by this ``I'') and so move
to the state $\mathtt{(I,right,0)}$,
\item ``N'', which starts a tree with codeword ``nN'' so the state becomes
$\mathtt{(N,left,1)}$,
\item ``N'' and put the pointer on the tree (just started by this ``N'') and so
move to the state $\mathtt{(N,right,1)}$.
\end{itemize}
 
 \begin{figure}[h!]
  \begin{center}
 \includegraphics[scale=0.6]{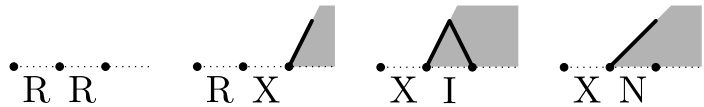}
  \end{center}
  \caption{If a diagram ends in ``R'' then we can either append another
``R'' or an ``X''. On the other hand, if the diagram ends in ``X'' then we must
start a tree with either ``I'' or ``N''.}
\label{fig RXappend}
\end{figure}

 \item If the upper forest ends in an ``R'' then it must be in state
$\mathtt{(R,right,0)}$. See Figure~\ref{fig RXappend}. We may append
\begin{itemize}
 \item ``R'' which keeps the diagram in state $\mathtt{(R,right,0)}$,
 \item ``X'' which adds a gap immediately to the left of a new tree, and moves to
$\mathtt{(X,right,0)}$.
\end{itemize}

\item If the upper forest ends in an ``X'' then it is in state $\mathtt{(X,right,0)}$ and
we must start a new tree. See Figure~\ref{fig RXappend}. We may append
\begin{itemize}
\item ``N'' and move to the state $\mathtt{(N,right,1)}$,
\item ``I'' and  move to the state $\mathtt{(I,right, 0)}$.
\end{itemize}

\begin{figure}[h!]
  \begin{center}
   \includegraphics[scale=0.6]{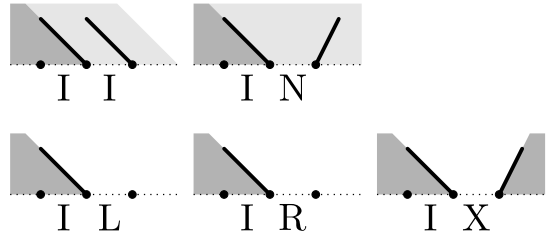}
  \end{center}
  \caption{If a diagram ends in ``I'' of excess 0 then we can continue the
tree with an ``I'' or an ``N'' as described in Section~\ref{ssec btrees} (note that an
internal node labelled by ``n'' will be created). If we are to the left of the pointer
then we can append an empty column labelled ``L''.  If we are to the right of the pointer
then we may append an empty column labelled ``R'' or ``X''.}
\label{fig Iappend}
\end{figure}

\item See Figure~\ref{fig Iappend}. If the diagram is in state $\mathtt{(I,left,0)}$ then
we may
\begin{itemize}
 \item continue the current tree by appending ``N'' or ``I'' and so move to
$\mathtt{(N,left,1)}$ or $\mathtt{(I,left,1)}$,
 \item finish the tree by appending ``L'' arriving in state $\mathtt{(L,left,0)}$.
\end{itemize}
\item Similarly, if the diagram is in state $\mathtt{(I,right,0)}$ then
we may
\begin{itemize}
 \item continue the current tree by appending ``N'' or ``I'' and so move to
$\mathtt{(N,right,1)}$ or  $\mathtt{(I,right,0)}$,
 \item finish the tree by appending ``R'' or ``X'' and so arrive in state
$\mathtt{(R,right,0)}$ or $\mathtt{(X,right,0)}$.
\end{itemize}

\item Finally, if the diagram is in state $\mathtt{(N,p,h)}, \mathtt{(I,p,h)}$ with $h>0$
and $\mathtt{p}=$ $\mathtt{left}$ or $\mathtt{right}$ then we are in the middle of constructing
a tree (as described in Section~\ref{ssec btrees}). Note that one cannot have label ``N''
with excess zero, nor may we change~$p$. We may now
\begin{itemize}
\item add ``N'' and increase the excess by 1, moving to
$\mathtt{(N,p,h+1)}$, 
\item add ``N'' leaving the excess unchanged giving
$\mathtt{(N,p,h)}$,
\item add ``I'' leaving the excess unchanged giving
$\mathtt{(I,p,h)}$
\item add ``I'', decreasing the excess by 1,  moving to
$\mathtt{(I,p,h-1)}$
\end{itemize}

\end{itemize}

Again, these constructions are easily turned into an algorithm and we refer the reader to
the pseudocode for \texttt{UpdateLeft(state)} and \texttt{UpdateRight(state)} in
Appendix~\ref{app B}.

One might also attempt to write out a set of recurrences, but they are quite cumbersome;
we now require several sets of coefficients since one must keep track of the excess,
the final label and whether one is to the left or right of the pointer. This
becomes extremely unwieldy for the full algorithm when one keeps track of the states of
the upper and lower forests. Because of this we were unable to write out the full
recurrences; instead we keep the recurrences implicit and iterate them using our
algorithms.

\subsection{The full algorithm}
The full algorithm is very similar to Algorithm~\textsc{CountBinaryTrees}, 
except that now the states of the upper and lower forests must be remembered and all
possible pairs of transitions must be determined. The set of pairs of transitions is very
nearly the Cartesian product of the set of transitions for the upper forest and the
set of transitions for the lower  forest; the only restriction is that we must prevent the
creation of common carets.

A caret is created in the upper (lower) forest immediately above (below) a gap when a
transition is made from a gap labelled ``I'', ``N'', or ``X'' to a gap labelled ``I''
(note that it is not possible to have a transition from ``R'' to ``I''). Thus a pair of
common carets is created when a transition is made from a pair of gaps labelled  $\II$
from a pair of gaps labelled by $\LL$, $\LN$, $\LX$, $\NL$, $\NN$, $\NX$, $\XL$, $\XN$ or
$\XX$.

We iterate through the recurrence starting from a diagram that consists of a single
empty column with gaps labelled by $\LL$. At each iteration we examine all the possible
states and consider all possible transitions from those states (avoiding common-carets)
to new states using the rules described above. The contribution of diagrams of the current
length is then added to the appropriate total for diagrams with new states with weight
updated according to the new labels and Table~1. A forest diagram is complete when it
finishes with an empty column labeled by $\RR$. This should be seen to be very similar
to Algorithm~\textsc{CountBinaryTrees} except that the transitions are more complicated.

This enumerates all forest diagrams starting with an empty column labelled by
$\LL$ and ending with another empty column labelled by $\RR$. This, unfortunately, is not
the precise set of diagrams that Belk \& Brown considered. They considered diagrams
without any empty columns at either end; Proposition~\ref{correct} shows how to
correct the output of the algorithm.

\vspace{3mm}

\noindent \hrulefill

\noindent \textbf{Algorithm B.} Counts forest diagrams starting with $\LL$ and ending with $\RR$.

\vspace{-3mm}
\noindent \hrulefill 

\noindent Input maximum length $N$ and set length $\ell=2$.

\vspace{3mm}

\noindent Set $\mathtt{totals(2,(L,left,0),(L,left,0))}=1$ --- this counter stores the
number of diagrams of a given weight ending in a column with given labels and
excess.

\vspace{3mm}

\noindent At the current length $\ell$ run through all the states of the upper and
lower forests, $\sigma,\tau$, for which $\mathtt{totals}(\ell,\sigma,\tau) \neq 0$.

\vspace{3mm}

\noindent Let $c = \mathtt{totals}(\ell,\sigma,\tau)$. This is the number of diagrams of
length $\ell$ that end in a column in state $\sigma, \tau$.

\vspace{3mm}

\noindent Given the current state of the upper forest compute the new states that may be
reached according to the rules in Section~\ref{ssec trans}. Do similarly for the lower
forest.

\vspace{3mm}

\noindent For each pair of transitions $(\sigma', \tau')$ that does not create a common
caret, compute the new length $\ell'$ using Table~1, and 
increment $\mathtt{totals}(\ell',\sigma',\tau')$ by $c$.

\vspace{3mm}

\noindent When all the states of the current length have been examined, output the number
of complete diagrams, namely $\mathtt{totals}(n,
\mathtt{(R,right,0)},\mathtt{(R,right,0)}))$.

\vspace{3mm}

\noindent Keep iterating until the desired length is reached.

\vspace{-3mm}
\noindent \hrulefill 

\vspace{3mm}

We give a more detailed account of this algorithm in pseudocode for
\texttt{CountForestDiagrams()} in Appendix~\ref{app B}. This main routine relies on the
routines \texttt{UpdateLeft()} and \texttt{UpdateRight()} to compute the possible
transitions. The cost analysis of the algorithm is given below in
Proposition~\ref{prop:timespace}.

As noted above the forest diagrams enumerated by this algorithm begin with any positive
number of empty columns labelled  ``$\LL$'' and then end in any positive number of empty
columns labelled ``$\RR$'', whereas those of Belk \& Brown have no such empty columns.
Thankfully we can easily obtain the correct count as follows.

Let \[F(q)=f_0+f_1q+f_2q^2+\ldots\] be the formal power series with coefficients $f_n=$
the number of forests of weight $n$ with {\em no} leading ``$\LL$'' or tailing ``$\RR$''
columns. Let \[H(q)=h_4q^4+h_5q^5+h_6q^6+\ldots\] be the formal power series with
coefficients $h_n=$ the number of forests of weight $n$ which have a positive number (at
least 1) leading ``$\LL$'' and a positive number of tailing ``$\RR$'' columns (which are
 counted by our algorithm). Note that $h_0=h_1=h_2=h_3=0$ since such forests must have
weight at least 4. Then
\begin{prop}
\label{correct}
\[F(q)=\left(\frac{1-q^2}{q^2}\right)^2H(q)\]
\end{prop}
\bp
Let $G(q)$ be the formal  power series with coefficients $g_n=$ the number of forests of
weight $n$ with no tailing ``$\RR$'' columns, but one or more leading ``$\LL$'' columns.
Note that $g_0=g_1=0$ since such forests have weight at least 2.  Then
$g_n=f_{n-2}+f_{n-4}+\ldots$ since an empty column labelled  ``$\LL$'' increases the
length of a forest diagram by $2$. That is, we have $f_{n-2}$ with just one leading
``$\LL$'' , $f_{n-4}$ with two leading ``$\LL$'' , and so on.

Thus
\begin{alignat*}{7}
 G(q) &= &  g_2q^2 & + & g_3q^3 & + & g_4q^4 &+ & g_5q^5 & + & g_6q^6 &+ & \ldots\\[2ex]
      &= &  f_0q^2 & + & f_1q^3 & + & f_2q^4 & + & f_3q^5 & + & f_4q^6 & + & \ldots\\
      &  &         &   &        & + & f_0q^4 & + & f_1q^5 & + & f_2q^6 & + & \ldots\\
      &  &         &   &        &   &        &   &        & + & f_0q^6 & + & \ldots
\end{alignat*}
\vspace{-5ex}
\begin{align*}
 \hspace{11ex} &= (q^2+q^4+q^6+\ldots) F(q) = \frac{q^2}{1-q^2} F(q)
\end{align*}
A similar argument shows that \[H(q)=G(q)\frac{q^2}{1-q^2} \] from which the result easily follows.
\ep

 \subsection{Cost analysis}
 
\begin{prop} \label{prop:timespace}
 Algorithm B computes the number of elements in the spheres of radius up to $n$ in time
$O(n^3)$ and space $O(n^2)$.
\end{prop}
\bp
The final generating function transformation takes constant time and space, so our main
concern is the complexity of enumerating a diagram with at most $n$ columns. Since each
column has weight at least one, each element of the ball of radius $n$ will be
represented by such a diagram.

When a new column is appended to a diagram the weight changes by at least 1 and
at most 4. Thus incomplete diagrams of weight $n$ will only contribute to
incomplete diagrams of weights between $n+1$ and $n+4$. Once we have scanned
through all diagrams of weight $n$ and computed all possible transitions we will
not need them again and they can be thrown away. In this way we only need to
store information about diagrams of five different weights at any given time.
Thus the algorithm needs to store $O(n^2)$ coefficients. During the
$n^\mathrm{th}$ iteration of the algorithm we scan through each of the $O(n^2)$ diagrams of
weight $n$ and compute the possible transitions; this takes $O(n^2)$ operations.
Since a diagram of weight $n$ may contain at most $n$ columns the algorithm
requires $O(n^3)$ operations.\ep

\subsection{Growth rate}
\begin{center}
\begin{table}[h!]
\begin{center}
\begin{tabular}{|c|c||c||c|}
 \hline
 $n$ & $f(n)$ & $f(n)^{1/n}$ & $f(n)/f(n-1)$ \\
 \hline
0 & 1 & $\circ$ & $\circ$ \\
1 & 4 & 4 & 4 \\
2 & 12 & 3.464\dots & 3 \\
3 & 36 & 3.301\dots & 3 \\
4 & 108 & 3.223\dots & 3 \\
5 & 314 & 3.157\dots & 2.907\dots \\
6 & 906 & 3.110\dots & 2.885\dots \\
7 & 2576 & 3.070\dots & 2.843\dots \\
8 & 7280 & 3.039\dots & 2.826\dots\\
9 & 20352 & 3.011\dots & 2.795\dots\\
10 & 56664 & 2.987\dots & 2.784\dots \\
\hline
20 & 1257901236 & 2.850\dots & 2.688\dots \\
\hline
50 & 6015840076078706884412 & 2.726\dots & 2.627\dots \\
\hline
100 & $\underbrace{5023\cdots5868}_{43\;\mathrm{digits}}$ & 2.673\dots &
2.6184\dots\\
\hline
200 & $\underbrace{3158\cdots3328}_{85\;\mathrm{digits}}$ & 2.645\dots
&2.618034\dots \\
\hline
500 & $\underbrace{7798\cdots8648}_{210\;\mathrm{digits}}$ & 2.628\dots &
2.6180339887498949\dots\\
\hline
1000 & $\underbrace{7579\cdots7676}_{419\;\mathrm{digits}}$ & 2.623\dots &
differs from $\frac{3+\sqrt{5}}{2}$ by about $10^{-32}$ \\
\hline
1500 & $\underbrace{7367\cdots9566}_{628\;\mathrm{digits}}$ & 2.62167\dots &
differs from $\frac{3+\sqrt{5}}{2}$ by about $10^{-48}$ \\
\hline
\end{tabular}\end{center}
\caption{Table of the number of elements of geodesic length $n$. By Fekete's
Lemma, $f(n)^{1/n}$ is an upper bound on the growth rate. We note that the
ratio of successive terms appears to converge rapidly to Guba's conjectured
value for the growth rate; unfortunately we have not been able to prove
that it converges.}
\end{table}
\end{center}

Using this algorithm we have been able to compute $f(n)$ for $0 \leq n \leq 1500$ using modest
computer resources. We give some of the resulting data in Table~2; the full sequence has
been published as sequence A156945 in the Online Encyclopedia of Integer
Sequences \cite{OEIS}. We coded the algorithm in \texttt{gnu-c++} using the Standard
Template Library \texttt{map} container to store the coefficients indexed by their
labels\footnote{We also implemented the algorithm using \texttt{hashmap}; this used
roughly the same amount of computer memory. Despite being theoretically faster than
\texttt{map} we found that it slowed the algorithm in practice.} and we used the
\texttt{CLN} number library\footnote{Available from {\texttt{http://www.ginac.de/CLN/}} at
time of writing.} to handle large integer arithmetic\footnote{An alternative approach
would have been to use modular arithmetic and reconstruct the results using the Chinese
remainder theorem. Since $f(n)$ grows exponentially with $n$, this approach would require
us to run the algorithm roughly $O(n)$ times using different primes. This would appear to
be a way  of adapting this algorithm to run on a cluster of computers.}.

As noted in the previous section, the sequence $f(n)$ is submultiplicative and so Fekete's
lemma \cite{Fekete} implies that the growth rate, $\lim_{n\to\infty} f(n)^{1/n}$, exists
and is bounded above by $f(n)^{1/n}$ for any $n$. Our data then gives a sequence of
rigorous upper bounds on the growth rate and combining with Guba's lower bound
\cite{Guba2004} gives the following result.

\begin{prop}
 The growth rate $\gamma:=\lim_{n \to \infty} f(n)^{1/n}$ of Thompson's group satisfies
\begin{align}
  \frac{3+\sqrt{5}}{2}=2.61803\ldots & \leq  \gamma \leq
f(1500)^{1/1500} =
2.62167\ldots
\end{align}
\end{prop}

We have plotted the sequence of upper bounds, $f(n)^{1/n}$, together with the lower bound
in Figure~\ref{fig boundplot} (left). It is not unreasonable to extrapolate these upper bounds
and reach the following conjecture:
\begin{conj}
 The growth rate of Thompson's group~$F$ with standard generating set
is~$\tfrac{3+\sqrt{5}}{2}$.
\end{conj}

\begin{figure}[h!]
 \begin{center}
 \bt{ccc}
 \includegraphics[height=5.5cm]{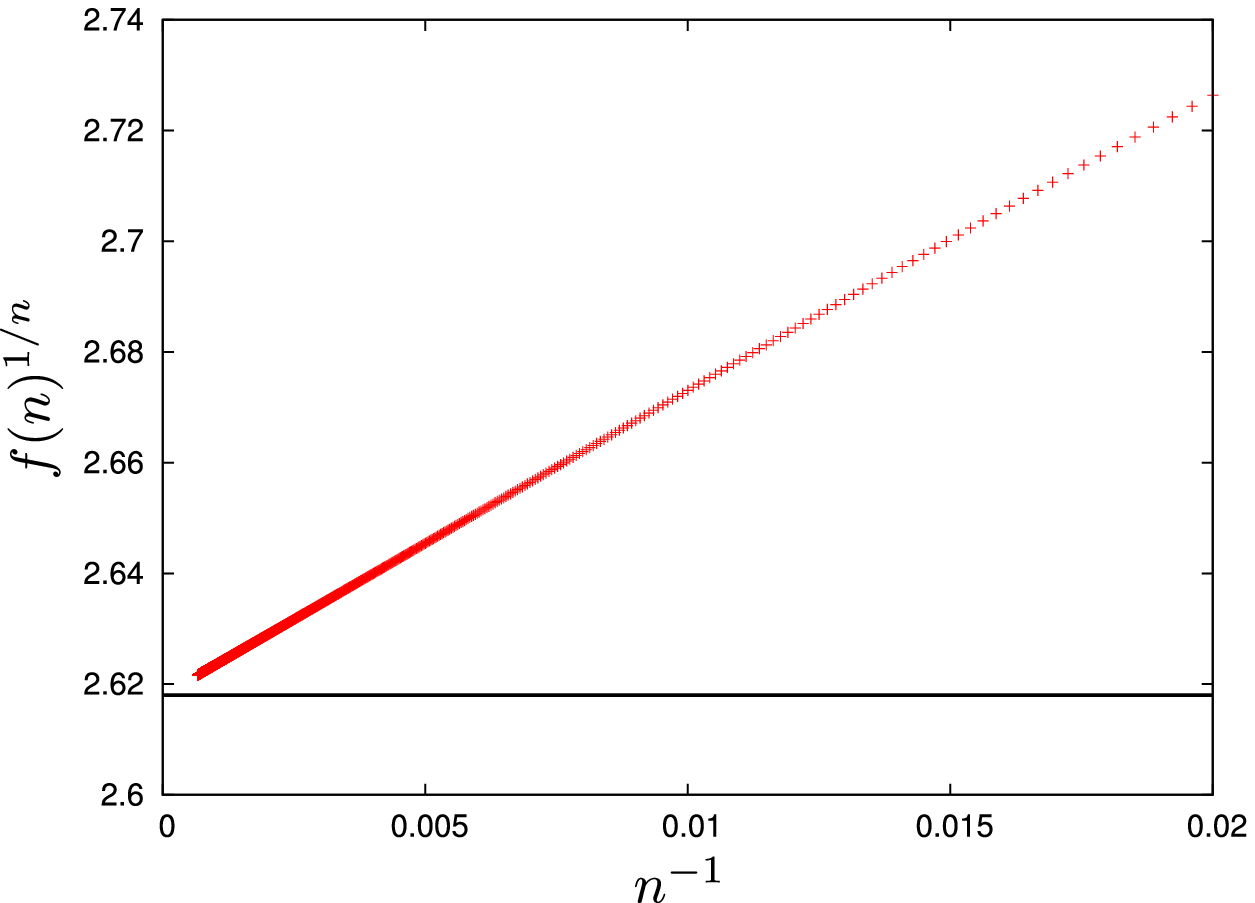} & $\;\;\;$ & \includegraphics[height=5.5cm]{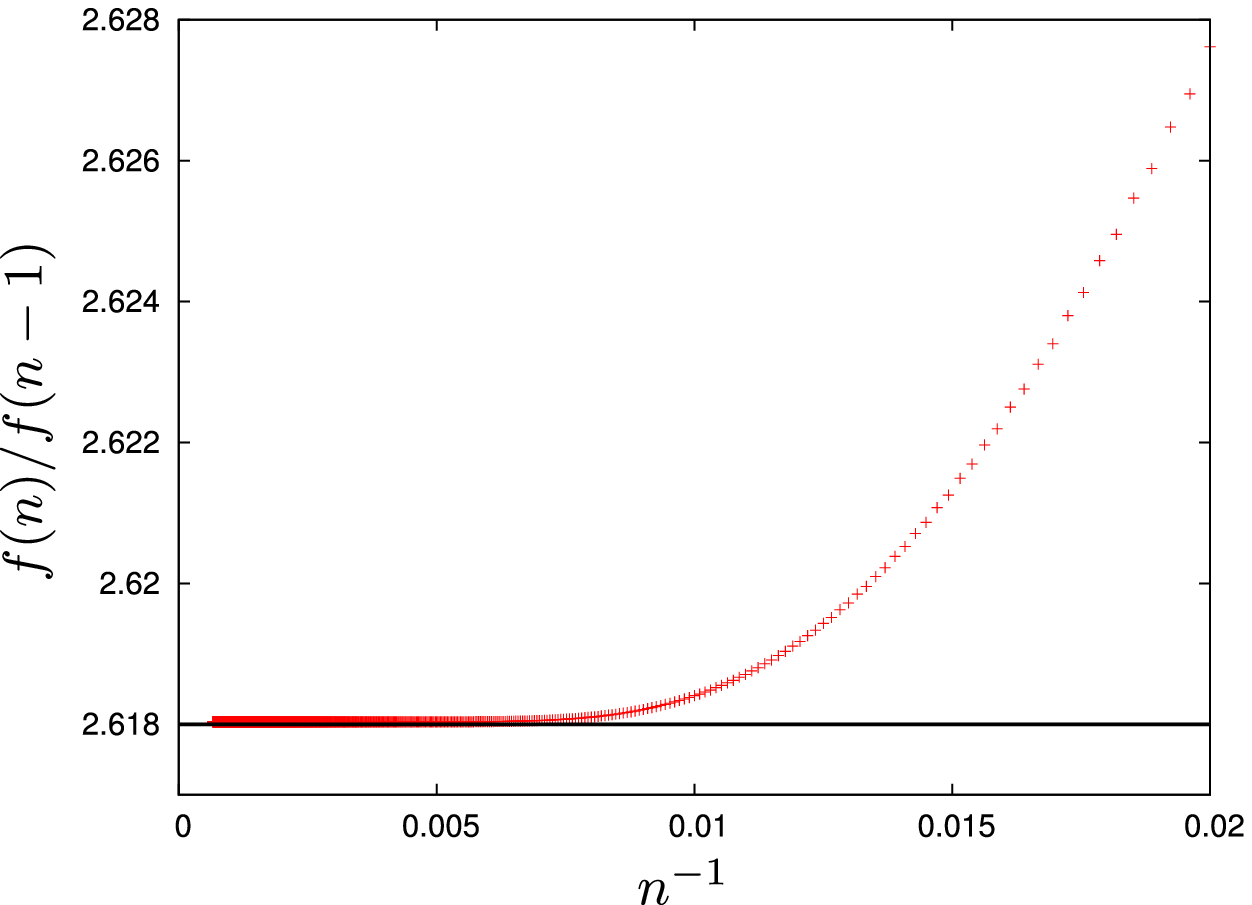} \\
\et
 \end{center}
  \caption{ (left) A plot of $f(n)^{1/n}$ versus $n^{-1}$  for $50 \leq n\leq 1500$.
For comparison we have also plotted Guba's lower bound of $\tfrac{3+\sqrt{5}}{2}$.
 (right) A plot of  the ratios $f(n)/f(n+1)$ for $50 \leq n\leq 1500$,
 and Guba's lower bound. For large $n$ the ratios are indistinguishable from Guba's lower
bound.}
\label{fig boundplot}
\end{figure}

To further support this conjecture we have examined the sequence $\left(f(2n)/f(n)\right)^{1/n}$
which must converge to the growth rate and we observe that it converges extremely rapidly to Guba's
lower bound. Additionally, in Figure~\ref{fig boundplot} (right) we have plotted
$f(n)/f(n-1)$. The limit of this ratio, if it exists, must be the growth rate. While we
have not been able to prove convergence it does appear to converge rapidly and is nearly
indistinguishable from Guba's bound. Note that Guba computed this ratio for $n=9$ in
\cite{Guba2004}, to obtain a conjectured upper bound of $2.7956043\ldots$.

Numerical analysis of the series \cite{Elder2009} 
 indicates that the corresponding generating function $F(z) = \sum f(n) z^n$ has an
isolated simple pole at $\frac{2}{3+\sqrt{5}}$ implying that
\begin{align}
 f(n) = A \left( \tfrac{3+\sqrt{5}}{2} \right)^n + \mbox{exponentially smaller
terms}
\end{align}
with $A \approx 8.02374\dots$. The first correction term appears to be
approximately $O(2.432^n)$. We estimated this term by studying the
asymptotics of the generating function $(1-3z+z^2)F(z)$; this polynomial factor
cancels the dominant contribution to the asymptotics from the (observed) simple pole at
$\frac{2}{3+\sqrt{5}}$.

\section{Outlook}\label{sec:outlook}
In this paper we have described two algorithms for computing the size of the sphere of
radius $n$ of Thompson's group with standard generating set. The first of these runs in
 exponential time and polynomial space, is potentially generalisable to other groups, and also computes the
geodesic growth series.
The second runs in polynomial time and space and we have used it to compute the first 1500 terms of the
growth series. We then used Fekete's lemma to compute upper bounds for the growth
rate from this data. 
This work suggests that the growth rate of~$F$ is exactly equal to $\frac{3+\sqrt{5}}{2}$.
This
strongly indicates that the normal forms described in \cite{Guba1997, Guba2004} have very
nearly geodesic length. More precisely, we believe that the word length of the
normal form of a typical element differs from its geodesic length only by $O(1)$.

We analysed the sequence $f(n)$ using our data and we have found some indication (using
series analysis techniques such as differential approximants \cite{Guttmann1989}) that
the corresponding generating function contains square-root singularities and so is
unlikely to be rational. Further, despite having the first fifteen hundred terms of the
sequence, we have been unable to conjecture a rational, algebraic or differentiably finite
generating function. In particular we used \texttt{Guess} package developed by
Manuel Kauers \cite{Kauers} to search through a wide range of possible
recurrences\footnote{The sequence in question does not satisfy any homogeneous linear
recurrence of order $r$ with polynomial coefficients of degree at most $d$, where $0\leq r
\leq R$ and $0 \leq d \leq D$ and $(R,D)$ taken from the list $(749,0), (165,7), (99,13),
(61,22), (52,26), (35,39), (43,32), (26,52), (23,60), (17,81), (10,134), (4,298)$. This is
a roughly exhaustive search of the all such recurrences that can be detected with the
first 1500 terms. Note that any rational, algebraic or differentiably finite sequence
must also satisfy a homogeneous linear recurrence with polynomial coefficients.}.
Unfortunately this search did not find any candidates. While this does not prove that the
generating function lies outside these classes,  it does rule out the possibility that the
generating function is simple (satisfying a recurrence of low degree or order).

It is relatively straightforward to extend Algorithm B to generate elements of a fixed
geodesic length uniformly at random in polynomial time. Algorithm A was derived from an
approximate enumeration algorithm \cite{GARM} which can also be used to sample large
elements of the group. We are currently investigating how these random generation
techniques might be used (in the same spirit as \cite{Burillo2007}) to explore properties
of the group that are still beyond analytic techniques such as its amenability.

\appendix
\section{Appendix: Pseudocode}
\label{appendix}
Throughout the appending we will use the following notations:
\begin{itemize}
  \item $x \leftarrow y$: set the variable $x$ to value $y$,
 \item $x = y$: the boolean operation that returns ``true'' if $x$ and $y$ are the
same and otherwise returns ``false'', and
  \item $x += y$: increment the variable $x$ by $y$.
\end{itemize}
We use $x \leftarrow y$ to distinguish the assignment of a value to a variable
from the test for equality. While we could use $x \leftarrow x+y$ instead of $x += y$ we
feel that the later easier to read when using long and descriptive variable names.
 
\subsection{Algorithm A}
\label{app A}
The function \texttt{ComputeGeod}  recursively outputs all geodesic words of length $N$.
\begin{algorithm}
 \caption{\texttt{ComputeGeod($w$,$N$)} --- output all geodesics of length~$N$
with prefix $w$. \label{alg computegeod}}
  \SetKwFunction{ComputeGeod}{ComputeGeod}
  \KwIn{Geodesic word $w$, Maximum length $N$}
  \lIf{$|w|=N$}{ \KwOut{$w$} }
  \Else{
  \lForEach{$x \in d_+(w)$}{ \ComputeGeod($wx$, $N$)}
  }
\end{algorithm}
Of course this function has the drawback that it generates a list of the geodesics that
must be stored either in memory or on disk before it can be processed to give the number
of elements. The following function avoids that problem.
\begin{algorithm}
 \caption{\texttt{ComputeSphere($w$,$N$)} --- Find the size of $S(n)$}
  \SetKwFunction{ComputeGeod}{ComputeGeod}
  \KwIn{Geodesic word $w$, Maximum length $N$}
  
  $s \leftarrow 0$;
  
  \If{$|w|=N$}{ $s \leftarrow \prod_{i=1}^n \frac{1}{|d_-(w_i)|}$ }
  \Else{
  \lForEach{$x \in d_+(w)$}{ $s += $\ComputeGeod($wx$, $N$)}
  }

  \KwOut{$s$}
\end{algorithm}
Calling \texttt{ComputeSphere($\epsilon, N$)}, where $\epsilon$ is the empty word, will
return the size of the sphere of radius $N$. It recursively computes the geodesics of
length $N$ starting with prefix $w$ and instead of storing them in a list, it returns the
sum of their contributions. Alternatively a non-recursive procedure
\texttt{NextGeodesic()} is given below. This computes the first geodesic after $w$ of
length at most $N$.

\begin{algorithm}
 \caption{\texttt{NextGeodesic($w$, $N$)} --- find the first geodesic after $w$ of length
  at most $N$}
 \KwIn{Geodesic word $w$, Maximum length $N$}
  \tcp{For this algorithm let $d_+(w) = \es$ when $|w|=N$.}

   \If{$d_+(w) \neq \es$}{
    $x \leftarrow $ first generator in $d_+(w)$
    
   \KwOut{$wx$}
   }

   \While{$d_+(w) = \es$}{
    $x \leftarrow $ last letter of $w$

    Delete last letter of $w$.

    \If{$x \neq $ last generator in $d_+(w)$}{
      $y \leftarrow $ first generator after $x$ in $d_+(w)$.
      
      \KwOut{$wy$}
    }
    \ElseIf{$|w|=0$}{ \KwOut{``No more geodesics''} }
   }
\end{algorithm}

\subsection{Algorithm B}
\label{app B}
We give the pseudocode for our main algorithm divided into three different functions.
Given the current state of the upper (or lower) forest diagram that is left of the
pointer, the function \texttt{UpdateLeft()} returns the possible states of the diagrams
produced by a valid transition (as described in Section~\ref{ssec trans} above).
Similarly, given the current state of the upper (or lower) forest diagram that is right
of the pointer, the function \texttt{UpdateRight()} returns the possible states of the
diagrams produced by a valid transition.

Finally, \texttt{CountForestDiagrams()} enumerates all forest diagrams according to their
weight (the geodesic length of the group elements they represent). It starts from the
empty diagram with gaps labelled by $\LL$. At each iteration it runs through all the
pairs of upper and lower states that have been reached and determines which pairs of
states can be reached by valid transitions using \texttt{UpdateLeft()} and
\texttt{UpdateRight()}. The transitions of the upper and lower forests are nearly
independent of each other; the only restriction is that we avoid creating common carets
and they are easily avoided. The weight of the diagram can then be updated using the
information in Table~1 and the appropriate counters can be updated. At the end of each
iteration we output the number of completed diagrams of the current weight (which are
those ending in gaps labelled $\RR$.

\begin{algorithm}
  \caption{\texttt{UpdateLeft(state)} --- Return the set of states that can be
reached from the current state when left of pointer}
  \label{alg update lor}
  \tcp{State is left of pointer by assumption}
  \KwIn{State of half-column $\mathtt{(label, left, h})$}

  \tcp{NewStates will be the set of states reached from the current state}
  NewStates $\leftarrow \es$.
  
  \If{$\mathtt{label} = \mathtt{L}$} {
    Add $\mathtt{(L, left, 0)}$ to NewStates \tcp*{another gap}
    
    Add $\mathtt{(N, left, 1)}, \mathtt{(I,left,0)}$ to NewStates \tcp*{start a tree}
    
    Add $\mathtt{(N, right, 1)}, \mathtt{(I,right,0)}$ to NewStates \tcp*{pointer \& start
    tree}
  
    Add $\mathtt{(R,right,0)}$ to NewStates \tcp*{pointer \& gap}
    
    Add $\mathtt{(X,right,0)}$ to NewStates \tcp*{pointer \& gap and start new tree next}
  }
  \If{$\mathtt{label} = \mathtt{N}$ or ($\mathtt{label}=\mathtt{I}$ and $\mathtt{h}>0$)}{
      Add $\mathtt{(N,left,h+1)}, \mathtt{(N,left,h)}, \mathtt{(I,left,h)}$ and
      $\mathtt{(I,left,h-1)}$ to NewStates
      \tcp*{Continue current tree}
    }
  \If{$\mathtt{label}=\mathtt{I}$ and $\mathtt{h}=0$}{
    Add $\mathtt{(N,left,1)}$ and $\mathtt{(I,left,0)}$ to NewStates \tcp*{continue tree}
    
    Add $\mathtt{(L,left, 0)}$ to NewStates \tcp*{finish current tree}

    \tcc{Note that one cannot finish a tree and immediately have the pointer,
    so the current state cannot be followed by a state labelled ``R'' or ``X''}
  }
  \KwOut{NewStates}
\end{algorithm}

\begin{algorithm}
  \caption{\texttt{UpdateRight(state)} --- Return the set of states that can be
reached from the current state when right of pointer}
  \label{alg update ror}
  \tcp{State is right of pointer by assumption}
  \KwIn{State of half-column $\mathtt{(label, right, h})$}

  \tcp{NewStates will be the set of states reached from the current state}
  NewStates $\leftarrow \es$.
  
  \If{$\mathtt{label} = \mathtt{R}$}{
    Add $\mathtt{(R,right,0)}$ to NewStates  \tcp*{another gap}
    
    Add $\mathtt{(X,right,0)}$ to NewStates \tcp*{another gap, start new tree next}
  }
  \If{$\mathtt{label} = \mathtt{X}$}{
    Add $\mathtt{(N,right,1)}$ and $\mathtt{(I,right,0)}$ to NewStates  \tcp*{start a
    tree}
  }
  \If{$\mathtt{label} = \mathtt{N}$ or ($\mathtt{label}=\mathtt{I}$ and $\mathtt{h}>0$)}{
      Add $\mathtt{(N,right,h+1)}, \mathtt{(N,right,h)}, \mathtt{(I,right,h)}$ and
      $\mathtt{(I,right,h-1)}$ to NewStates
    \tcp*{Continue current tree}
  }
  \If{$\mathtt{label}=\mathtt{I}$ and $\mathtt{h}=0$}{
    Add $\mathtt{(N,right,1)}$ and $\mathtt{(I,right,0)}$ to NewStates \tcp*{continue
    tree}
    
    Add $\mathtt{(R,right, 0)}$ to NewStates \tcp*{finish current tree}

    Add $\mathtt{(X,right, 0)}$ to NewStates \tcp*{finish current tree, start new tree
    next}
  }
  \KwOut{NewStates}
\end{algorithm}

\begin{algorithm}
 \caption{\texttt{CountForestDiagrams($M$)} --- Count forest diagrams of weight at most
$M$.}
 \label{alg poly main}
  \SetKwFunction{UpdateLeft}{UpdateLeft}
  \SetKwFunction{UpdateRight}{UpdateRight}
  \SetKwFunction{Weight}{Weight}

 \KwIn{Maximum length $M$}

  \tcc{$\mathtt{totals(n,\sigma,\tau)}$ stores the number of diagrams of weight $n$,
  with upper diagram in state $\sigma$ and lower diagram in state $\tau$. Initially all
  are zero except the following}
 
  $\mathtt{totals(2,(L,left,0),(L,left,0))} \leftarrow 1$

  \For{$n \leftarrow 2$ \KwTo $M-1$}{
  \ForEach{$(\sigma,\tau)$ with $\mathtt{totals(n,\sigma,\tau)} \neq 0$} {
  \tcp{The following produces sets of new upper and lower forests}
    \If{$\sigma$ is left of upper pointer} {
      upper-set $\leftarrow$ \UpdateLeft{$\sigma$}
    }
    \Else{
      upper-set $\leftarrow$ \UpdateRight{$\sigma$}
    }
    \If{$\tau$ is left of lower pointer} {
      lower-set $\leftarrow$ \UpdateLeft{$\tau$}
    }
    \Else{
      lower-set $\leftarrow$ \UpdateRight{$\tau$}
    }
    
    \tcp{Construct new diagrams in states $\sigma',\tau'$ from all possible pairs of
  transitions}
    \ForEach{ $(\sigma', \tau') \in$ upper-set $\times$ lower-set} {
    \tcp{We must check new columns for common carets.}
    \If{ ($\sigma'$.label $= \tau'$.label = $\mathtt{I}$) and
    ($\sigma$.label $\neq \mathtt{I}$) and
    ($\tau$.label $\neq \mathtt{I}$) }{
	\tcp{reject new state as it produced common caret}
    }
    \Else{
    \tcp{no common caret so keep new state with updated weight}
      \tcp{\Weight{} computes the change in weight using Table~1}
      $\mathtt{totals}(n$+\Weight{$\sigma'$.label, $\tau'$.label}, $\sigma',\tau') +=
\mathtt{totals}(n,\sigma,\tau)$
    }
    }
    \KwOut{ $(n+1, \mathtt{totals}(n+1, \mathtt{(R,right,0)},\mathtt{(R,right,0)}))$ }
    }
    }
\end{algorithm}


\begin{thebibliography}{00}
\bibliographystyle{elsarticle-num}



\bibitem{BelkBrown2003}
J.M. Belk and K.S. Brown.
\newblock {Forest diagrams for elements of Thompson's group~$F$}.
\newblock {\em Internat. J. Algebra Comput.}, 15:815--850, 2005.

\bibitem{BousquetMelou}
M.~{Bousquet-M\'elou}.
\newblock {A method for the enumeration of various classes of column-convex
  polygons}.
\newblock {\em Disc. Math.}, 154(1-3):1--25, 1996.

\bibitem{MBMMP_2000}
M.~{Bousquet-M\'elou} and M.~Petkov\v{s}ek.
\newblock{Linear recurrences with constant coefficients: the multivariate case}.
\newblock {\em Disc. Math.} 225: 51--75, 2000.

\bibitem{MBMMP_2003}
M.~{Bousquet-M\'elou} and M.~Petkov\v{s}ek.
\newblock{Walks confined in a quadrant are not always D-finite}.
\newblock {\em Theoret. Comput. Sci.} 307: 257--276, 2003.


\bibitem{Burillo2007}
J.~Burillo, S.~Cleary, and B.~Wiest.
\newblock Computational explorations in {Thompson's} group {$F$}.
\newblock In {\em {Geometric Group Theory, Geneva and Barcelona Conferences}}.
  Birkhauser, 2007.

\bibitem{CFP}
J.~W. Cannon, W.~J. Floyd, and W.~R. Parry.
\newblock Introductory notes on {R}ichard {T}hompson's groups.
\newblock {\em Enseign. Math. (2)}, 42(3-4):215--256, 1996.

\bibitem{CombCT}
S.~Cleary and J.~Taback.
\newblock Combinatorial properties of Thompson's group~$F$.
\newblock {\em Trans. AMS.}, 356(7):2825--2849 (electronic), 2004.



\bibitem{3choice}
A.R. Conway, A.J. Guttmann, and M.~Delest.
\newblock {The number of three-choice polygons}.
\newblock {\em  Math. Comput. Modelling}, 26:51--58, 1997.



\bibitem{ElderGeodesics}
M.~Elder and A.~Rechnitzer.
\newblock {Some geodesic problems for finitely generated groups}.
\newblock {\em Arxiv preprint arXiv:09}, 2009.

\bibitem{Elder2009}
M.~Elder, {\'{E}.}~Fusy and A.~Rechnitzer.
\newblock {Experimenting with Thompson's group~$F$}.
\newblock In preparation -- title subject to change.

\bibitem{Enting1985}
I.G. Enting and A.J. Guttmann.
\newblock {Self-avoiding polygons on the square, L and Manhattan lattices}.
\newblock {\em J. Phys. A: Math. Gen.}, 18(6):1007--1017, 1985.

\bibitem{Felsner2008}
S.~Felsner, {\'{E}.}~Fusy, M.~Noy, and D.~Orden.
\newblock {Bijections for Baxter Families and Related Objects}.
\newblock {\em Arxiv preprint arXiv:0803.1546}, 2008.

\bibitem{Fordham2003}
S.B. Fordham.
\newblock {Minimal Length Elements of Thompson's Group~$F$}.
\newblock {\em Geom. Dedicata}, 99(1):179--220, 2003.

\bibitem{GrigTat}
R.~Grigorchuk and T.~Smirnova-Nagnibeda.
\newblock Complete growth functions of hyperbolic groups.
\newblock {\em Inven. Math.}, 130(1):159--188, 1997.

\bibitem{Guba1997}
V.S. Guba and M.V. Sapir.
\newblock {The Dehn function and a regular set normal forms for R. Thompson's
  group~$F$}.
\newblock {\em J. Aust. Math. Soc. Series A}, 62:315--328, 1997.

\bibitem{Guba2004}
V.S. Guba.
\newblock {On the Properties of the Cayley Graph of Richard Thompson's Group~$F$}.
\newblock {\em Int. J. of Alg. Computation}, 14(5-6):677--702, 2004.

\bibitem{Guttmann1989}
A.J. Guttmann.
\newblock {Asymptotic analysis of power-series expansions}.
\newblock {\em Phase Transitions and Critical Phenomena}, 13:1--234, 1989.

\bibitem{PMMC}
J.M. Hammersley and K.W. Morton.
\newblock {Poor man's Monte Carlo}.
\newblock {\em J. Roy. Statist. Soc. B}, 16(1):23--38, 1954.

\bibitem{Kauers}
M.~Kauers.
\newblock {Guess --- a Mathematica package for guessing multivariate and univariate
recurrence equations.}
\newblock{Available from}\\
\newblock{\texttt{http://www.risc.uni-linz.ac.at/research/combinat/software/} }




\bibitem{Matucci}
F.~Matucci.
\newblock {Algorithms and Classification in Groups of Piecewise-Linear Homeomorphisms}.
\newblock {\em Arxiv preprint arXiv:0807.2871}, 2008.



\bibitem{GARM}
A.~Rechnitzer and E.J. Janse~van Rensburg.
\newblock {Fast Track Communication}.
\newblock {\em J. Phys. A: Math. Theor.}, 41:442002, 2008.

\bibitem{Rosenbluth}
M.N. Rosenbluth and A.W. Rosenbluth.
\newblock {Monte Carlo calculation of the average extension of molecular
  chains}.
\newblock {\em J. Chem. Phys.}, 23(2):356--362, 1955.

\bibitem{OEIS}
N. J. A. Sloane.
\newblock  The On-Line Encyclopedia of Integer Sequences, (2008)
\newblock \texttt{http://www.research.att.com/$\sim$njas/sequences/}


\bibitem{StanleyECV2}
R.P. Stanley.
\newblock {\em {Enumerative Combinatorics: Volume 2}}.
\newblock Cambridge University Press, 1999.

\bibitem{Temperley}
H.N.V. Temperley.
\newblock {Combinatorial Problems Suggested by the Statistical Mechanics of
  Domains and of Rubber-Like Molecules}.
\newblock {\em Phys. Rev.}, 103(1):1--16, 1956.

\bibitem{Fekete}
J.H. van Lint and R.M.~Wilson.
\newblock {\em {A Course in Combinatorics}}.
\newblock Cambridge University Press, 2001.

\end{thebibliography}
\end{document}